\documentclass{amsart}

\usepackage[T1]{fontenc}
\usepackage{avant}
\usepackage[cp1250]{inputenc}
\usepackage{amsmath,amssymb,amsthm}

\newcommand\Hom{\operatorname{Hom}}
\newcommand\End{\operatorname{End}}

\newtheorem{theorem}{Theorem}[section]
\newtheorem{proposition}[theorem]{Proposition}
\newtheorem{lemma}[theorem]{Lemma}
\newtheorem{definition}[theorem]{Definition}

\usepackage{todonotes}

\usepackage{mathtools} 
\usepackage{stmaryrd} 

\usepackage{tikz}
\usepackage[all]{xy}

\usepackage{hyperref} 

\theoremstyle{definition}

\def\R{\mathbb{R}}
\def\C{\mathbb{C}}

\renewcommand\Im{\operatorname{Im}}

\DeclareMathOperator\id{id}

\title[Complex structures on the Six-Sphere]{Chern's contribution to the Hopf problem: an exposition based on Bryant's paper}
\author{Aleksy Tralle and Markus Upmeier}

\begin{document}

\maketitle 

\begin{abstract}
We give a comprehensive account of Chern's Theorem that $S^6$ admits no $\omega$-compatible almost complex structures. No claim to originality is being made, as the paper is mostly an expanded version of material in differential sources already in the literature.
\end{abstract}

\setcounter{tocdepth}{1}
\tableofcontents

\section{Introduction}

Following Bryant's exposition~\cite{Br}, we present a theorem of Chern that there is no \emph{integrable} almost complex structures on $S^6$ compatible with the standard 2\nobreakdash-form $\omega$ on $S^6$. 
It is determined by the octonionic almost complex structure $J_{can}$, see \eqref{def-standard-ac}, and the round metric $g_{can}$ on $S^6$ through
\[
\omega(u,v) \coloneqq g_{can}(J_{can}u,v).
\]


\begin{definition}
An almost complex structure $J$ on $S^6$ is {\it $\omega$-compatible} if
\begin{equation}\label{omega-compatible}
\omega(u,v)=\omega(Ju,Jv)\qquad \forall u,v.
\end{equation}
\end{definition}

We have left out the usual condition $\omega(u,Ju)>0$ for $u\neq 0$. Instead, the \emph{$\omega$-index} of $J$ is defined as the index $(2p,2q)$ of the non-degenerate symmetric bilinear form $g_J\coloneqq \omega(\cdot,J\cdot)$.
The main result is as follows.

\begin{theorem}[Chern]\label{thm:Chern}
There are no $\omega$-compatible complex structures on $S^6$.
\end{theorem}

The reader may also wish to refer to \cite{Da} for a related proof.

\section{The exceptional Lie group $G_2$}

\subsection{$G_2$-action on $S^6$}

For the understanding of this paper, a shortcut definition of the exceptional Lie group $G_2$ suffices. More information may be found in \cite{Ag}.
Let $V\coloneqq \R\oplus \C^3$ with basis $e_1,\ldots, e_7$ and the standard inner product. We identify $S^6$ with the unit sphere in $V$. Define a basis of the complexification $V_\C = \C^7$ by
\begin{equation}\label{standard-basis}
	e_1, F_1=\frac{e_2-ie_3}{2},F_2=\frac{e_4-ie_5}{2},F_3=\frac{e_6-ie_7}{2}, \bar{F}_1, \bar{F}_2, \bar{F}_3.
\end{equation}
Then $|e_1|=1$ and $|F_k|=|\bar{F}_k|=1/\sqrt{2}$. We use this basis to identify endomorphisms of $V_\C$ with matrices. Let
	\begin{equation}\label{G2Lie}
\mathfrak{g}_2 = \left\{
\begin{pmatrix}
	0 & -ia^* & ia^t\\
	-2ia&D&[\bar{a}]\\
	2i\bar{a}& [a] & \bar{D}
	\end{pmatrix}
	\middle|
	a \in \C^3, D\in \mathfrak{su}(3)
	\right\} \subset \C^{7\times 7},
\end{equation}
using the notation $a^*=\bar{a}^t$ and
\[
	[a] \coloneqq \begin{pmatrix}
0&a_3 & -a_2\\
-a_3 & 0 & a_1\\
a_2 & -a_1 & 0
\end{pmatrix} \in \C^{3\times 3},\quad
a\in \C^3.
\]

Then $\mathfrak{g}_2 \subset \mathfrak{so}(V)\subset \mathfrak{su}(V_\C)$ since by normalizing $F_k$ to unit length the matrix in \eqref{G2Lie} becomes skew-Hermitian. 
It is easy to check that \eqref{G2Lie} is closed under the matrix Lie bracket.
According to Lie's Theorems there exists a unique simply-connected Lie group $G_2$ with this Lie algebra  and a smooth monomorphism $G_2 \to SO(V)$. Since the Killing form is negative definite on $\mathfrak{g}_2$, the group $G_2$ is compact so that
\begin{equation}\label{G2faith}
G_2 \subset SO(V)
\end{equation}
is topologically embedded.
Using \eqref{standard-basis} we write this faithful representation as
\begin{align*}
&\begin{gathered}
\xymatrix{
G_2\ar[r]^-\rho\ar[rd]_-{\rho^\C} & SO(V)\cong SO(7)\ar[d]^{-\otimes\C}\\
& SU(V_\C) = SU(\C^7)
}
\end{gathered}
&\rho_g&=(g_1,\ldots, g_7),
&\rho^\C_g &= (x,f_1,f_2,f_3,\bar{f}_1,\bar{f}_2,\bar{f}_3),
\end{align*}
using column notation for the matrices $\rho_g, \rho_g^\C$. Thus $g_i = \rho_g(e_i)$ and $x=g_1$, $f_1=\rho_g^\C(F_1)=\frac12(g_2 - ig_3)$ and so on. The functions $x,f_i,\bar{f_i}\colon G_2 \to \C^7$ are called the \textbf{moving frame} on $G_2$.

Restricting to unit vectors, \eqref{G2faith} defines a smooth $G_2$-action on $S^6$ and $x$ is simply the orbit map at $e_1 \in S^6$. To proceed, we next need the differential of $x$.

\subsection{Structure Equations}

\begin{definition}\label{def:MC}
The \emph{Maurer--Cartan form} $\phi \in \Omega^1(G_2; \mathfrak{g}_2)$ is the matrix-valued form $\phi = g^{-1}dg$. Thus $\phi(X \in T_g G_2)=g^{-1}\cdot X$ \textup(matrix multiplication\textup).
\end{definition}

The wedge product of matrix-valued differential forms is given by the usual formula, using matrix multiplication instead of the product of numbers. In terms of \eqref{G2Lie} we write the components of $\phi$ as
\[
	\phi = \begin{pmatrix}
	0 & -i\theta^* & i\theta^t\\
	-2i\theta&\kappa&[\bar{\theta}]\\
	2i\bar{\theta}& [\theta] & \bar{\kappa}
	\end{pmatrix},\quad
	\theta \in \Omega^1(G_2; \C^3), \kappa \in \Omega^1(G_2; \mathfrak{su}(3)).
\]

\begin{theorem}[Bryant~\cite{BrOct}] \label{thm:str-eq}
We have the \textbf{first structure equations} \textup(where $f=(f_1,f_2,f_3)$ in the obvious vector notation\textup)
\begin{equation}
\label{G2MC}
d(x,f,\bar{f})=(x,f,\bar{f}) \cdot 
\begin{pmatrix}
	0 & -i\theta^* & i\theta^t\\
	-2i\theta&\kappa&[\bar{\theta}]\\
	2i\bar{\theta}& [\theta] & \bar{\kappa}
	\end{pmatrix}.
\end{equation}
Also, the \textbf{second structure equations} hold:
\begin{align}
&d\theta = -\kappa\wedge \theta + [\bar\theta]\wedge\bar\theta\label{second-structure1}\\
&\begin{aligned}
d\kappa &= -\kappa\wedge \kappa + 2\theta\wedge\theta^* - [\bar\theta]\wedge[\theta]
\end{aligned}
\end{align}
\end{theorem}

\begin{proof}
In our matrix notation $g=(x,f,\bar{f})$. So \eqref{G2MC} is just Definition~\ref{def:MC} multiplied by $g$. The second structure equations follow by reading off matrix entries on both sides of the so-called Maurer--Cartan equation
\[
	d\phi = d(g^{-1}dg)=-g^{-1}dg\wedge g^{-1}dg = -\phi\wedge \phi.\qedhere
\]
\end{proof}

\subsection{$S^6$ as a homogeneous space}


\begin{lemma}
The action of $G_2$ of $S^6$ is transitive with isotropy group $SU(3)$. Hence the orbit map $x=\rho_g(e_1)$ restricts to a principal $SU(3)$-bundle
\begin{equation}\label{principalSU}
	x\colon G_2 \to S^6,
\end{equation}
where $SU(3)$ is embedded in $G_2\subset SU(\C^7)$ as
\begin{equation}\label{SUembedding}
	\begin{pmatrix}
	1&&\\
	&A&\\
	&&\bar{A}
	\end{pmatrix},\qquad
	\forall A \in SU(3).
\end{equation}
\end{lemma}
\begin{proof}
By \eqref{G2MC} the differential is, where the notation indicates a matrix-vector multiplication $f\cdot \theta=f_1\theta_1 + f_2\theta_2  + f_3\theta_3$,
\begin{equation}\label{dx}
dx=-2if\cdot \theta +2i\bar{f}\cdot\bar\theta.
\end{equation}
Hence $x$ is a submersion. The image is therefore open and closed, so all of $S^6$. By the
long exact sequence of homotopy groups of the fibration $u$ together with the fact that $G_2$ is connected and $S^6$ is simply-connected, the stabilizer must be simply connected and is therefore $SU(3)$.
\end{proof}

\begin{lemma}
For the \textbf{right translation} we have
\begin{align}\label{right-translation}
	R_A^* \theta &= A^{-1}\cdot \theta
	&\forall A&\in SU(3).
\end{align}
\end{lemma}
\begin{proof}
We have
\[
(R_A^*\phi)(X) = (gA)^{-1} X\cdot A = A^{-1}\cdot \phi(X)\cdot A\qquad \forall X \in T_gG_2.
\]
Now perform the matrix multiplication and compare entries in
\[
	R_A^*\phi_{G_2} = 
	\begin{pmatrix}
	1&0&0\\
	0&A^{-1}&0\\
	0&0&A^t
	\end{pmatrix}
	\cdot
	\begin{pmatrix}
	0 & -i\theta^* & i\theta^t\\
	-2i\theta&\kappa&[\bar{\theta}]\\
	2i\bar{\theta}& [\theta] & \bar{\kappa}
	\end{pmatrix}
	\cdot
	\begin{pmatrix}
	1&0&0\\
	0&A&0\\
	0&0&\bar{A}
	\end{pmatrix}.\qedhere
\]
\end{proof}

\section{The standard almost complex structure on $S^6$}


Let $y\in S^6$. Fix also a $g\in G_2$ with $x(g)=y$. The submersion \eqref{principalSU} induces an exact sequence
\begin{equation}\label{exact-G2S6-sequence}
	0\to T_g(gSU(3)) \to T_g G_2 \xrightarrow{dx_g} T_y S^6  \to 0.
\end{equation}

According to \eqref{dx}, the forms $\theta^i_g, \bar\theta^i_g$ vanish on the kernel of $dx$ and thus descend to a basis of $T_y^* S^6 \otimes \C$. 
Note the dependence on $g$, but by \eqref{right-translation} the spanned subspaces $\langle \theta^1,\theta^2,\theta^3 \rangle$ and $\langle \bar\theta^1,\bar\theta^2,\bar\theta^3 \rangle$ are invariant under $SU(3)$ and hence determine a well-defined subspace of $T_yS^6\otimes \C$. We may therefore define:

\begin{definition}
The octonionic complex structure $J_{can}$ is defined for any choice of $g\in G_2$ with $x(g)=y$ by the decomposition
\begin{align}\label{def-standard-ac}
	T^{1,0}_{J_{can}} (T^*_yS^6) &= \langle \theta^1_g,\theta^2_g,\theta^3_g \rangle,
	&T^{0,1}_{J_{can}} (T^*_yS^6) &= \langle \bar\theta^1_g,\bar\theta^2_g,\bar\theta^3_g \rangle.
\end{align}
\end{definition}

This is in fact a nearly K\"ahler structure (see also \cite{Da}):

\begin{proposition} \label{prop:str-eq-cor} There exists a complex 3-form $\Upsilon$ on $S^6$ such that
\begin{enumerate}
\item $x^*g_{can}=4\theta^t\circ\bar\theta$ where $g_{can}$ denotes also the $\C$-bilinear extension of the round metric to $TS^6\otimes \C$.
\item $x^*\omega=2i\theta^t\wedge\bar\theta$
\item $d\omega=-3\operatorname{Im}(\Upsilon)$
\item $x^*\Upsilon=8\theta^1\wedge\theta^2\wedge\theta^3$ and $\Upsilon$ has $J_{can}$-type $(3,0)$.
\end{enumerate}
\end{proposition}
\begin{proof}
i) Both sides are $G_2$-invariant, so we check equality at $1\in G_2$. Write $A \in \mathfrak{g}_2$ as in \eqref{G2Lie}. Then using $|f_k|=1/\sqrt{2}$
\begin{align*}
g_{can}(dx(A), dx(A)) &= g_{can}(-2iaf+2i\bar{a}\bar{f},-2iaf+2i\bar{a}\bar{f})\\&= 4\|a\|^2 = 4(\theta^t \circ \bar\theta)(A,A) 
\end{align*}
(the $(1,0)$ and $(0,1)$-subspaces are isotropic for the 
$\C$-bilinear extension.)\\
ii) By i) $\sqrt{2}\theta^i$ is an orthonormal basis of $(1,0)$-forms. Hence by~\eqref{omega-local}
\[
	x^*\omega = i\sqrt{2}\theta^i \wedge \sqrt{2} \bar\theta^i = 2i \theta^t \wedge \bar\theta.
\]
iv) $\theta_1\wedge \theta_2 \wedge \theta_3$ is invariant under $SU(3)$ since by \eqref{right-translation} for $A^{-1}=(a^{ij})$
\begin{align*}
	&(a^{11}\theta_1 + a^{12}\theta_2 + a^{13}\theta_3)
	\wedge
	(a^{21}\theta_1 + a^{22}\theta_2 + a^{23}\theta_3)
	\wedge
	(a^{31}\theta_1 + a^{32}\theta_2 + a^{33}\theta_3)\\
	=& \det(A^{-1}) \theta_1 \wedge \theta_2 \wedge \theta_3 = \theta_1 \wedge \theta_2 \wedge \theta_3.
\end{align*}
This proves the existence of $\Upsilon$. It is clearly a $(3,0)$-form.\\
iii) Using ${(\alpha\wedge \beta)^t=(-1)^{\alpha\beta}\beta^t\wedge\alpha^t}$, $\kappa^t=-\bar\kappa$, $[\theta]^t=-[\theta]$ and \eqref{second-structure1} we find
\begin{align*}
d(\theta^t \wedge \bar\theta) &= (d\theta)^t\wedge\bar\theta - \theta^t \wedge \overline{d\theta}\\
&= (-\kappa\wedge \theta + [\bar\theta]\wedge\bar\theta)^t\wedge \bar\theta - \theta^t \wedge (-\bar\kappa\wedge\bar\theta + [\theta]\wedge\theta)\\
&= \theta^t\wedge\kappa^t\wedge \bar\theta - \bar\theta^t\wedge[\bar\theta]^t\wedge\bar\theta + \theta^t\wedge\bar\kappa\wedge\bar\theta - \theta^t\wedge[\theta]\wedge \theta\\
&= -\theta^t\wedge [\theta]\wedge \theta + \overline{\theta^t\wedge [\theta]\wedge \theta}\\
&= 6(\theta_1\wedge\theta_2\wedge\theta_3 - \overline{\theta_1\wedge\theta_2\wedge\theta_3}) = 12i \Im(\theta_1\wedge\theta_2\wedge\theta_3)\qedhere
\end{align*}
\end{proof}

\section{Proof of Chern's Theorem}

\subsection{Comparision of almost complex structures}

Let $J$ be an arbitrary almost complex structure on $S^6$. Let
$$GL(3,\mathbb{C})\hookrightarrow F_J\xrightarrow{\pi} S^6$$
be the corresponding bundle of $J$-complex frames, with fiber at $y\in S^6$
\[
	u\in \pi^{-1}(y) = \Hom_\C((T_y S^6,J) , (\C^3,i)).
\]
We have a tautological $\mathbb{C}^3$-valued 1-form $\eta\in\Omega^1(F_J,\mathbb{C}^3)$ 
$$\eta(v)=u(d\pi(v)),\qquad \forall v\in T_u F_J.$$
Define $B_J$ by the pullback diagram
\[\xymatrix{
	B_J\ar[r]\ar[d]&F_J\ar[d]^\pi\\
	G_2\ar[r]_x & S^6,\\
}\]
a principal $SU(3)\times GL(3,\mathbb{C})$-bundle over $S^6$ whose elements are pairs $(g,u)$ with $x(g)=\pi(u)$. We have two submersions $B_J\rightarrow G_2$ and $B_J\rightarrow F_J$ along which we pull back the differential forms $\eta$ and $\theta$ to $B_J$.

\begin{proposition} There are unique smooth maps $r,s: B_J\rightarrow \C^{3\times 3}$ with
\begin{equation}\label{theta-eta}
\theta=r\eta+s\bar\eta,\quad\Rightarrow\quad\bar\theta=\bar s\eta+\bar r\bar\eta.
\end{equation}
Also, the matrix
$\begin{pmatrix}
r & s\\
\bar s & \bar r
\end{pmatrix}
$
has non-zero determinant.
\end{proposition}

\begin{proof}
Let $(g,u)\in B_J$ with $y=x(g)=\pi(u)$. Thus $u\colon T_yS^6 \to \C^3$ is a $(J,i)$-complex linear isomorphism. In particular,
\begin{equation}\label{ubasis}
	u^1,u^2,u^3,\bar{u}^1,\bar{u}^2,\bar{u}^3 \in (T_y^*S_6)\otimes \C
\end{equation}
is a complex basis of $T^{1,0}_J (T^*S^6) \oplus T^{0,1}_J (T^*S^6)$. Hence we have
an expansion $\theta = ru+s\bar{u}$. Recall from \eqref{def-standard-ac} that $\theta^i_g, \bar\theta^i_g$ is a second basis of $(T^*S^6)\otimes \C$. We thus get a change of basis matrix with non-zero determinant.
\end{proof}

\subsection{The bundles $\mathbb{J}_1(\omega,S^6)$ and $\mathbb{J}_2(\omega,S^6)$}

Let $\mathbb{J}(M,\omega) \subset \End(TM)$ be the bundle of $\omega$-compatible almost complex structures on a smooth manifold $M$. Its fiber at $p\in M$ are all $J\colon T_pM \to T_pM$ with $J^2=-1$ and satisfying \eqref{omega-compatible} for $\omega|_{T_pM}$. Then
$$\mathbb{J}(M,\omega)=\cup_q\mathbb{J}_q(M,\omega)$$
where $\mathbb{J}_q(M,\omega)\subset\mathbb{J}(M,\omega)$ is the subbundle of almost complex structures of $\omega$-index $(2n-2q,2q)$.
Here the dimension of $M$ is $2n$. Thus, in the case $M=S^6$ we get
$$\mathbb{J}(S^6,\omega)=\mathbb{J}_0(S^6,\omega)\cup\mathbb{J}_1(S^6,\omega)\cup\mathbb{J}_2(S^6,\omega)\cup\mathbb{J}_3(S^6,\omega).$$

Because $S^6$ is connected an the $\omega$-index is a continuous pair of integers, every $\omega$-compatible almost complex structure $J$ is a section of one of these subbundles. Two cases can be ruled out topologically:

\begin{lemma}\label{lemma:j1-j2} $\mathbb{J}_1(\omega,S^6)$ and $\mathbb{J}_2(\omega,S^6)$ do not admit a global continuous section.
\end{lemma}
\begin{proof}
Assume that $J$ is a continuous section of $\mathbb{J}_1(S^6,\omega)$ or of $\mathbb{J}_2(S^6,\omega)$. Then the positive and negative definite subspaces of $g=\omega(\cdot,J\cdot)$ yield a decomposition
$$TS^6=E_4\oplus E_2$$
into two vector subbundles of ranks $4$ and $2$. However, as is well know, the Euler class (and characteristic) of $TS^6$ is nontrivial: $e(TS^6)\not=0$. On the other hand, rank $4$ and rank $2$ vector bundles over $S^6$ have trivial Euler classes, since $H^2(S^6,\mathbb{Z})=H^4(S^6,\mathbb{Z})=0$. Using the formula for the Euler classes of the Whitney sum one obtains the contradiction
\[
0\not=e(TS^6)=e(E_4)\cup e(E_2)=0.\qedhere
\]
\end{proof}

\subsection{Chern's identity}

Putting \eqref{theta-eta} into Proposition~\ref{prop:str-eq-cor},~ii) and using that $\eta$ has $J$-type $(1,0)$ gives
\begin{equation}\label{omega11}
\omega_J^{1,1}=2i \eta^t\wedge ( r^t\bar r- \bar s^t s)\bar\eta.
\end{equation}
The assumption that $J$ is $\omega$-compatible means that $\omega=\omega_J^{1,1}$ has $J$-type $(1,1)$.
%

\begin{proposition}
For any integrable $\omega$-compatible complex structure on $S^6$ we have \textbf{Chern's identity}
\begin{equation}\label{ChernIdentity}
\det (\bar s)=\det (r).
\end{equation}
\end{proposition}
\begin{proof}
Putting \eqref{theta-eta} into Proposition~\ref{prop:str-eq-cor},~iii) gives
\begin{equation}\label{Upsilon}
\Upsilon_J^{3,0}=8\det (r)\eta_1\wedge\eta_2\wedge\eta_3,\qquad
\Upsilon_J^{0,3}=8\det (s)\bar\eta_1\wedge\bar\eta_2\wedge\bar\eta_3.
\end{equation}
When $J$ is integrable, Lemma~\ref{lem:integrable} implies that $d\omega$ has type $(2,1)+(1,2)$. Hence its $(3,0)$-part with respect to $J$ vanishes. Recalling also $d\omega=3\operatorname{Im}(\Upsilon)$ we calculate
\begin{align*}
0&=(d\omega)^{3,0}_J=(3\operatorname{Im}(\Upsilon))^{3,0}_J=\frac{3}{2i} \left(\Upsilon-\bar\Upsilon\right)^{3,0}_J\\
&\overset{\eqref{Upsilon}}{=}12i\left(\det(\bar s)-\det (r)\right)\eta_1\wedge\eta_2\wedge\eta_3\qedhere
\end{align*}
\end{proof}

\subsection{Proof of Chern's Theorem}

Before giving the proof, recall that for two Hermitian matrices $A, B$ we say $A>B$ (resp.~$A\geq B$) if $A-B$ has only positive (resp.~non-negative) eigenvalues. $A>B$ is equivalent to
\begin{equation}\label{AgB} 
\langle Ax, x\rangle > \langle Bx,x \rangle,\qquad \forall x\neq 0.
\end{equation}
For example, for an arbitrary matrix $C$ we have $C^*C\geq 0$ and $CC^*\geq 0$. Moreover $C^*C>0$ and $CC^*>0$ precisely when $C$ is invertible.

\begin{lemma}\label{lem:determinant}
For $A>B>0$ we have $\det A > \det B > 0$.
\end{lemma}
\begin{proof}
By replacing $x=A^{-1/2}y$ in \eqref{AgB} we see that $A>B>0$ is equivalent to $E > A^{-1/2}BA^{-1/2}>0$ for the identity matrix $E$. Let $C\coloneqq A^{-1/2}BA^{-1/2}>0$ have eigenvalues $\lambda_i>0$. Then $E-C>0$ has eigenvalues $1-\lambda_i > 0$. Hence $\det(A^{-1})\det(B)=\det(C)\in (0,1)$.
\end{proof}

\begin{proof}[Proof of Theorem~\textup{\ref{thm:Chern}}]
Assume by contradiction that $J$ is both integrable and $\omega$-compatible. Then Lemma~\ref{lemma:j1-j2} shows that $J$ must be a section of $\mathbb{J}_0(S^6,\omega)$ or of $\mathbb{J}_3(\omega, S^6)$. Hence the bilinear form $\omega(\cdot,J\cdot)$, which according to \eqref{omega11} is represented by twice the matrix $H\coloneqq r^t\bar r-\bar s^t s$, is either positive definite or negative definite.

Assume that $H$ is positive definite, so $r^t\bar{r}>\bar{s}^ts$. Since $\bar s^t s\geq 0$ this implies
$r^t\bar{r} \geq H > 0$ and so $r^t\bar{r}$ is invertible. Hence $0\neq \det(r) = \det(\bar s)$
by \eqref{ChernIdentity} and so $r^t\bar{r}>\bar{s}^ts>0$, contradicting Lemma~\ref{lem:determinant}. The case when $H$ is negative definite is analogous. So we have reached a contradiction in every case.
\end{proof}

\begin{appendix}

\section{Preliminaries}

\subsection{Linear algebra}

Let $(V, J)$ be a complex vector space. The complexification $V_\C\coloneqq V\otimes_\R \C$ carries two commuting complex structures $J_\C \coloneqq J\otimes\id_\C$, $i=1\otimes i$ which gives a splitting into the $(\pm i)$-eigenspaces of $J_\C$
\begin{equation}\label{splitting}
	V_\C = V^{1,0}\oplus V^{0,1}.
\end{equation}
By convention $V_\C, V^{1,0}, V^{0,1}$ are equipped with the complex structure $i$. We identify $V\hookrightarrow V_\C$ by $v\mapsto v\otimes 1$ with image the \emph{real subspace}, the subspace of $V_\C$ fixed by complex conjugation $\bar{\;}\colon V_\C\to V_\C$.

These definitions apply to $V=\R^{2n}$ or to the tangent space $T_pM$ at a point of an almost complex manifold $(M,J)$. Note that complex structures on $V$ are equivalent to complex structures $J^*$ on the dual space. Then $(V^*)^{1,0}$ is isomorphic to the $(J,i)$-complex linear maps $V \to \C$ and similarly $(V^*)^{0,1}$ are the complex anti-linear maps. However, both determine complex-linear maps $V_\C$ with respect to the complex structure $i$. We may decompose the complex $n$-forms as
\[
	\Lambda^n(V^*) = \bigoplus_{p+q=n}\Lambda^{p,q}(V^*),\qquad \Lambda^{p,q}(V^*)=\Lambda^p(V^*)^{1,0} \otimes \Lambda^q(V^*)^{0,1}
\]
and we denote the corresponding projection by $\alpha \mapsto \alpha^{p,q}$.

Conversely, a splitting $V_\C=V^{1,0}\oplus V^{0,1}$ of the complexification of a real vector space $V$ into two complex subspaces satisfying $\overline{V^{1,0}}=V^{0,1}$ defines a unique complex structure on $V$ with given type decomposition: decompose $v\in V$ as $v\otimes 1=v^{1,0} + v^{0,1}$ and define $J(v)=iv^{1,0} - iv^{0,1}$ (which again belongs to the real subspace).

Now suppose that $g$ is a Euclidean metric on $V$ such that $J$ is $g$-orthogonal. Then we obtain a Hermitian form on $V_\C$ by
\[
	h(v_1\otimes z_1, v_2\otimes z_2) \coloneqq g(v_1,v_2)\otimes z_1\overline{z_2},
\]
for which $V^{1,0}\oplus V^{0,1}$ is orthogonal. One may also complexify $g$ to a real $\C$-bilinear form $g_\C(v_1\otimes z_1, v_2\otimes z_2)=g(v_1,v_2)\otimes z_1z_2$ for which $V^{1,0}$ and $V^{0,1}$ are isotropic. Since $J$ is skew-symmetric for $g$ we have also a $2$-form on $V$
\begin{equation}
\omega(X,Y)=g(JX,Y).
\end{equation}
Let $\{z_\alpha\}_{\alpha=1,\ldots, \dim_\C V}$ be a complex basis of $V^{1,0}$ with dual basis $z^\alpha$. Then $\bar{z}_\alpha$ is a basis of $V^{0,1}$. Letting $h_{\alpha\bar\beta}\coloneqq h(z_\alpha,z_\beta)=g_\C(z_\alpha,\bar{z}_\beta)$ the complexification of $\omega$ is
\begin{align}\label{omega-local}
 \omega &= ih_{\alpha\bar\beta} z^\alpha\wedge \bar{z}^\beta.
\end{align}

\subsection{Almost complex manifolds}

An almost complex structure is an endomorphism $J\colon TM\to TM$ satisfying $J^2=-1$.
For example, an complex manifold is almost complex, since the derivative of local holomorphic coordinates
gives real linear isomorphisms $\C^n \to T_pM$ along which we may transport the standard complex structure $i$ to get $J$. An almost complex structure of this type is called \emph{integrable}.

Let $\mathcal{A}^{p,q}(M)$ be the global sections of the bundle $\Lambda^{p,q}(T^*M)$.

\begin{lemma}\label{lem:integrable}
Suppose $J$ is integrable and let $\eta \in \mathcal{A}^{p,q}(M)$. Then
\[
d\eta \in \mathcal{A}^{p+1,q}(M)\oplus \mathcal{A}^{p,q+1}(M).
\]
\end{lemma}

\begin{proof}
When $J$ is integrable we may use the coordinates to get an exact local frame $dz^\alpha, d\bar{z}^{\bar\beta}$ of the $(1,0)$ and $(0,1)$-forms. By definition a $(p,q)$-form has a local expression
\[
	\eta = \eta_{\alpha_1\cdots\alpha p, \bar\beta_1\cdots\bar\beta_q} dz^{\alpha_1}\cdots dz^{\alpha_p} d\bar{z}^{\bar\beta_1}\cdots d\bar{z}^{\bar\beta_p}.
\]
Now apply $d$ and the fact that for a complex-valued function $f$ we have a splitting $df = \frac{\partial f}{\partial z^\alpha} dz^\alpha + \frac{\partial f}{\partial \bar{z}^\alpha} d\bar{z}^\alpha$ into the complex linear and anti-linear part.
\end{proof}

The converse of the lemma is the difficult Newlander-Nirenberg Theorem.

\subsection{Lie groups}

\begin{theorem}[Lie's Second Theorem]\label{LieSecondTheorem}
Let $G,H$ be Lie groups with $G$ simply connected. Taking the derivative at the unit sets up a bijection between Lie homomorphisms $G\to H$ and Lie algebra homomorphisms $\mathfrak{g}\to\mathfrak{h}$.
\end{theorem}

\begin{theorem}[Lie's Third Theorem]\label{LieThirdTheorem}
For every finite-dimensional real Lie algebra $\mathfrak{g}$ there exists a unique simply-connected Lie group $G$ whose Lie algebra is $\mathfrak{g}$. Any connected Lie group with that Lie algebra is isomorphic to $G/\Gamma$ for a discrete subgroup $\Gamma \subset Z(G)$ of the center.
\end{theorem}

\end{appendix}

\end{document}